\documentclass[11pt]{amsart}

\usepackage{amsthm,amsmath,amssymb,amscd,amsfonts}

\theoremstyle{plain}
\newtheorem{theorem}{Theorem}[section]

\newtheorem{proposition}[theorem]{Proposition}

\newtheorem{corollary}[theorem]{Corollary}
\newtheorem{def-thm}[theorem]{Definition-Theorem}
\newtheorem{lemma}[theorem]{Lemma}

\theoremstyle{definition}
\newtheorem{definition}[theorem]{Definition}

\newtheorem{remark}[theorem]{Remark}

\newtheorem{conjecture}[theorem]{Conjecture}

\newcommand{\PP}{\mathbb{P}}

\newcommand{\RR}{\mathbb{R}}

\newcommand{\NN}{\mathbb{N}}
\newcommand{\ZZ}{\mathbb{Z}}
\newcommand{\CC}{\mathbb{C}}
\newcommand{\QQ}{\mathbb{Q}}

\newcommand{\OO}{{\mathcal O}}

\newcommand{\II}{{\mathcal I}}

\DeclareMathOperator{\ord}{ord}

\DeclareMathOperator{\divisor}{div}

\DeclareMathOperator{\image}{image}
\DeclareMathOperator{\Pic}{Pic}

\allowdisplaybreaks
\begin{document}

\title[On essentially large divisors]{On essentially large divisors}

\author{Gordon Heier}
\author{Min Ru}

\address{Department of Mathematics\\University of Houston\\4800 Calhoun Road, Houston, TX 77204\\USA}

\email{heier@math.uh.edu}
\email{minru@math.uh.edu}

\subjclass[2000]{11G35, 11G50, 14C20, 14G40, 32H30}

\begin{abstract}
Motivated by the classical Theorems of Picard and Siegel and their generalizations, we define the notion of an {\it essentially large} effective divisor and derive some of its geometric and arithmetic consequences. We then prove that on a nonsingular projective variety $X$ whose codimension is no greater than $\dim X-2$, every effective divisor with $\dim X +2$ or more components in general position is essentially large. 
\end{abstract}

\maketitle

\section{Introduction}
In \cite{cz_cr}, Corvaja and Zannier found an innovative way of using Schmidt's Subspace Theorem to give a new proof of the classical Theorem of Siegel on integral points on affine varieties. They subsequently expanded their approach to obtain certain results on integral points in higher dimensions (\cite{cz_crelle}, \cite{cz_annals}, \cite{cz_ajm}, \cite{cz_imrn}). The approach was translated to Nevanlinna theory in \cite{ru_ajm}.\par
Evertse and Ferretti (\cite{ferretti_compositio}, \cite{ef_imrn}) used similar (yet more general) arguments to obtain diophantine inequalities on projective varieties. Their approach is largely based on Mumford's degree of contact. In Nevanlinna theory, this approach was used by Ru \cite{ru_annals} to establish a Second Main Theorem for holomorphic curves into projective varieties intersecting hypersurfaces.\par
As stated in \cite[p.\ 609]{levin_annals}, the article \cite{cz_cr} motivated Levin to make the following Definition. Its significance is due to the hyperbolicity-type and Mordell-type properties enjoyed by the complements of large divisors. We let $K$ denote either a number field or $\CC$.
\begin{definition}[{\cite[Definition 8.1]{levin_annals}}]
Let $D$ be an effective divisor on a nonsingular projective variety $X$ defined over $K$. Then $D$ is said to be {\it very large} if for every $P\in D(\bar K)$, there exists a basis $B$ of the finite-dimensional vector space
$$L(D)=\{f \text{\ rational function on\ } X| \divisor(f) \geq -D\}$$
such that $\ord_E \prod_{f\in B} f > 0$ for every irreducible component $E$ of $D$ with $P\in E(\bar K)$. Moreover, an effective divisor is said to be {\it large} if it has the same support as some very large divisor.
\end{definition}
Recall that for an effective divisor $D$ as in the above Definition, the finite-dimensional vector spaces $L(D)$ and $H^0(X,\OO_X(D))$ are isomorphic via
\begin{equation}H^0(X,\OO_X(D))\to L(D),\quad s\mapsto \frac s {s_D},\label{h0_isom}\end{equation}
where $s_D$ is a section of $\OO_X(D)$ with $\divisor (s) =D$.\par
For a divisor $D=\sum_{i=1}^{r} D_i$ with $D_i$ effective (but not necessarily irreducible), Levin's main sufficient criterion for very largeness is the following.
\begin{lemma}[{\cite[Lemma 9.1]{levin_annals}}]\label{Levin_9_1} Let $D=\sum_{i=1}^{r} D_i$ be a divisor with each $D_i$ effective (but not necessarily irreducible) on a nonsingular projective variety $X$, all defined over $K$. For $P\in D(\bar K)$, let $D_P=\sum_{\{i:P\in D_i\}}D_i$. Also, for integers $m,n$, let
$$f_P(m,n)=h^0(X,\OO_X(nD-mD_P))-h^0(X,\OO_X(nD-(m+1)D_P)).$$
If there exists $n>0$ such that $\sum_{m=0}^\infty (m-n)f_P(m,n) > 0$ for all $P\in D(\bar K)$, then $nD$ is very large.
\end{lemma}
The short proof provided in \cite{levin_annals} is based on the filtration argument introduced in \cite{cz_cr}. The main idea is to define a filtration $$V_j=H^0(X,\OO_X(nD-jD_P)) \quad (j=0,1,2,3,\ldots)$$
of $H^0(X,\OO_X(nD))$ and to choose a basis $f_1,\ldots ,f_{h^0(X,\OO_X(nD))}$ of $L(nD)$ according to this filtration, beginning with the last nonzero subspace. Since $\dim V_j/V_{j+1}=f_P(j,n)$, we get
$$\ord_E \prod_{i=1}^{h^0(X,\OO_X(nD))} f_i \geq (\ord_E D)\sum_{m=0}^{\infty} (m-n)f_P(m,n) > 0.$$\par
The main result on large divisors in \cite{levin_annals} is the following.
\begin{theorem}[\cite{levin_annals}]\label{levin_2mq_thm}
Let $X$ be a $q$-dimensional nonsingular projective variety, defined over $K$. Let $D=\sum_{i=1}^{r} D_i$ be a divisor, also defined over $K$, with each $D_i$ effective (but not necessarily irreducible) and big and nef. Moreover, assume that every irreducible component of $D$ is nonsingular and that the intersection of any $m+1$ distinct $D_i$ is empty over $\bar K$. If $r>2mq$, then $D$ is large.
\end{theorem}
In the proof of this Theorem given in \cite{levin_annals}, the difference $h^0(X,\OO_X(nD-mD_P))-h^0(X,\OO_X(nD-(m+1)D_P))$ is bounded from above based on the $H^0$-part of the corresponding long exact cohomology sequence. As remarked by Levin, using the double filtration argument from \cite[Lemma 3.2]{cz_annals}, the factor of $2$ in the lower bound on $r$ can be removed.\par
It would of course be interesting to decrease the lower bound on $r$ in Theorem \ref{levin_2mq_thm}. In light of the classical Theorems of Picard and Siegel and their generalizations, and some of the standard conjectures in hyperbolicity theory, one might suspect that when $m=q$, the lower bound of $2mq$ can be replaced by $q+1$. However, it is quite clear that the sufficient criteria in \cite{levin_annals} cannot be used to prove this. To be precise, we make the following Remark.
\begin{remark}\label{P_q_suff_crit_remark}
In the situation of Lemma \ref{Levin_9_1}, let $X=\PP^q$. Let $D$ be the sum of $r$ general smooth hypersurfaces of degree $d$. Then
$$\sum_{m=0}^\infty (m-n)f_P(m,n) > 0\quad (n \text{\ sufficiently large})$$
if and only if $r> q^2+q$. 
\end{remark}
We leave the proof of Remark \ref{P_q_suff_crit_remark} as an exercise to the reader. Note that since $\Pic(\PP^q)$ is generated by $\OO_{\PP^q}(1)$ and since $h^0(\PP^q,\OO_{\PP^q}(\ell))={q+\ell\choose \ell}$, the dimensions of all the $H^0$'s involved can be computed without problem.\par
We now turn to the new results presented in this note. The starting point is our definition of {\it essentially (very) large} divisors as follows.
\begin{definition}\label{ess_v_l_def}
Let $D$ be an effective divisor on a nonsingular projective variety $X$, all defined over $K$. Then we define $D$ to be {\it essentially very large} if there exists a linear subspace $V$ of $L(D)$ such that for every $P\in D(\bar K)$, there exists a basis $B$ of $V$ such that $\ord_E \prod_{f\in B} f > 0$ for every irreducible component $E$ of $D$ with $P\in E(\bar K)$. Moreover, we define an effective divisor to be {\it essentially large} if it has the same support as some essentially very large divisor.
\end{definition}
This is obviously a weakening of the notion of large divisors, which requires $V=L(D)$. However, our definition seems technically easier to handle, while it retains the important function theoretic and number theoretic consequences of the original definition of large divisors derived in \cite{levin_annals}. In fact, in Section \ref{section_nevan}, we prove new inequalities of Second Main Theorem-type and Schmidt Subspace Theorem-type for essentially large divisors that are actually stronger than the consequences for large divisors derived in \cite{levin_annals}. Since the statements of our Theorems require some (standard) definitions, we simply refer the reader to Theorem \ref{our_smt} and Theorem \ref{our_schmidt} at this point of the Introduction.\par
As a matter of convention, we will always assume that the constant function $1$ is an element of the basis $B$, although it does not contribute to the vanishing of the product.
\begin{remark}
The reader might wonder how essential very largeness fits in with the standard notions of ``size'' for a divisor in algebraic geometry. In this paper, we will not make any statements in this respect, other than the trivial observation that a basis for $V$ yields a nonconstant rational map to projective space whose locus of indeterminacy (i.e., the base locus of $V$ understood as a subspace of $H^0(X,\OO_X(D))$) is contained in $D$.
\end{remark}
In a second line of inquiry, beginning in Section \ref{section_ess_v_l_proofs}, it is clearly interesting to ask for (sharp) sufficient criteria for a divisor to be essentially large. In this direction, we first discovered that the filtration method from \cite{cz_ajm} (see also \cite{ru_ajm}) yields the expected sharp bound (even for the original version of largeness) in the case of $X=\PP^q$ as stated in the following Theorem. (Cf. our comments in Remark \ref{P_q_suff_crit_remark}.) 
\begin{theorem}\label{thm_1}
Let $q\geq 1$ and $r\geq q+2$ be integers. On $\PP^q$, let $D=\sum_{i=1}^r D_i$ be a divisor defined over $K$, where each $D_i$ is a hypersurface (not necessarily irreducible or reduced). Assume that the $D_i$ are in general position, i.e., the intersection of any $q+1$ of them is empty over $\bar K$. Then $D$ is large.
\end{theorem}
In the case of nonsingular  projective varieties different from projective space, our result is the following.
\begin{theorem}\label{thm_2}
Let $q\geq 1$ and $r\geq q+2$ be integers. Let $X\subseteq \PP^\ell$ be a nonsingular projective variety of dimension $q$ defined over $K$. Let $D=\sum_{i=1}^r D_i$ be a divisor on $X$ such that each $D_i$ is defined by the restriction to $X$ of a homogeneous polynomial of degree $d_i$ in $K[X_0,\ldots,X_\ell]$. Finally, assume that the $D_i$ are in general position on $X$, i.e., the intersection of $X$ with any $q+1$ of them is empty over $\bar K$. Then $D$ is essentially large.
\end{theorem}
While the assumptions in Theorem \ref{thm_2} are certainly restrictive, they correspond exactly to the current state-of-the-art for generalized geometric versions of the Second Main Theorem (see \cite{ru_annals}).\par
It seemed to be an interesting problem to find a broad class of projective varieties to which Theorem \ref{thm_2} can actually be applied. This is treated in Section \ref{section_alg_geom}, where we obtain the following Theorem, which is repeated as Corollary \ref{cor}.
\begin{theorem}\label{cor_intro}
Let $q\geq 1$ and $r\geq q+2$ be integers. Let $X\subseteq \PP^\ell$ be a nonsingular projective variety of dimension $q$, defined over $K$. Assume that $2q - \ell \geq 2$ holds. Let $D=\sum_{i=1}^r D_i$ be an effective divisor on $X$ defined over $K$ such that the $D_i$ are in general position. Then $D$ is essentially large.
\end{theorem}
In closing this Introduction, we conjecture that the bound on the number of components in Theorem \ref{cor_intro} is generally the correct one, regardless of the codimension of $X$:
\begin{conjecture}\label{conjecture}
Let $q\geq 1$ and $r\geq q+2$ be integers. Let $X\subseteq \PP^\ell$ be a nonsingular projective variety of dimension $q$, defined over $K$. Let $D=\sum_{i=1}^r D_i$ be an effective divisor, defined over $K$, on $X$ such that the $D_i$ are big and in general position. Then $D$ is essentially large.
\end{conjecture}
\section{The geometric and arithmetic properties of essentially large divisors}\label{section_nevan}
\subsection{The geometric case}\label{geom_case}
We first derive the inequality of Second Main Theorem-type for essentially very large divisors. We begin by recalling some standard definitions from Nevanlinna Theory. \par
Let $g: {\CC} \rightarrow {\PP}^m$ be a  
holomorphic map. 
Let $g=[g_0, \dots, g_m]$ be a reduced representative of $g$, 
where $g_0, \dots, g_m$ are  entire functions 
on  ${\CC}$ and have no 
common zeros. The Nevanlinna-Cartan characteristic function (or the
height function) $T_g(r)$ 
is defined by
$$ T_g(r) = {1\over 2\pi}\int_0^{2\pi}\log  \max_{j=0,\ldots,m} 
|g_j(re^{\sqrt{-1}\theta})| d\theta.$$
The above definition is independent, up to an additive constant, of the choice of 
the reduced representation of $g$. \par
For a divisor $D$ on a projective variety $X$, represented by a local defining function $\rho$, and a holomorphic map
$g:\CC\to X$, 
the counting function is defined as 
$$N_g(r,D)=\int_0^ r {n_g(t,D)-n_g(0,D)\over t} dt + n_g(0,D)\log r,$$
where $n_g(t,D)$ is the number of zeros of $\rho\circ g$ inside
$|z|<t$ (counting multiplicities), and $n_g(0,D)=\lim_{t\to 0^+}n_g(t,D)$.\par
By the Jensen formula, we also have
$$N_g(r,D)=\sum_{a\in D_r\setminus \{0\}} \ord_a(\rho\circ g)\log
\left|\frac r a \right|+
\ord_0(\rho\circ g) \log r.$$\par
The following generalized version of Cartan's Second Main Theorem (see \cite{ru_gen_form_smt_tams}, \cite{vojta_smt_ajm_1997}) will be the basis of the proof of Theorem \ref{our_smt}. Note that by $\|\ \|$ we mean the (coefficient-wise) maximum norm in this Subsection.
\begin{theorem}\label{gen_smt}
Let $f=[f_0:\ldots:f_m]:\CC\to \PP^m$ be a holomorphic map whose image is not contained  in a proper linear subspace. Let $H_1,\ldots,H_q$ be arbitrary hyperplanes in $\PP^m$. Let $L_j$, $1\leq j \leq q$, be linear forms defining $H_1,\ldots,H_q$. Then, for every $\varepsilon>0$,
$$\int_0^{2\pi}\max_K \log \prod_{j\in K} \frac{\|f(r e^{\sqrt{-1}\theta})\|\|L_j\|}{|L_j(f(re^{\sqrt{-1}\theta})|}\frac{d\theta}{2\pi}.\leq. (m+1+\varepsilon)T_f(r),$$
where ``$.\leq.$'' means that the inequality holds for all $r$ outside of a set $\Gamma$ with finite Lebesgue measure, and the maximum is taken over all subsets $K$ of $\{1,\ldots,q\}$ such that the linear forms $L_j$, $j\in K$, are linearly independent.
\end{theorem}
To prove our result, we need a general formula for the height $T_g(r)$
when  $g=[g_0:\ldots:g_m]$, where $g_0, \dots, g_m$ are meromorphic
functions (i.e., it is not necessarily the {\it reduced} representation
of $g$). According to \cite[p. 202]{lang_book}, such a formula reads
\begin{align*}T_g(r)=&\int_0^{2\pi} \log \max_{j=0,\ldots,m}|g_j(re^{\sqrt{-1}\theta})|\frac{d\theta}{2\pi}-\log\max_{j\in A} |c_{g_j}|\\
&+\sum_{a\in D_r\setminus{0}}\max_{j=0,\ldots,m}(-\ord_a(g_j))\log\left|\frac r a \right|+\max_{j=0,\ldots,m}(-\ord_0(g_j))\log |r|,
\end{align*}
where $\log\max_{j\in A} |c_{g_j}|$ is a correction term that makes
the definition independent of multiplication by a nowhere zero
holomorphic function (see \cite{lang_book} for full details). \par
Our main result in this Subsection is the following inequality of Second Main Theorem-type. According to the subsequent Corollary \ref{hyp_cor}, it can be understood from the geometric point of view as a quantitative version of the Chern hyperbolicity of the complement of an essentially large divisor.
\begin{theorem}\label{our_smt}
Let $D$ be an essentially very large divisor on a nonsingular complex projective variety $X$. Let $V$ be a linear subspace of $L(D)$ as in Definition \ref{ess_v_l_def}. Let $\phi_0,\ldots\phi_m$ be an arbitrary basis for $V$. Let $f:\CC\to X$ be an algebraically nondegenerate holomorphic map. Let $\Phi$ be the rational map $[\phi_0:\ldots:\phi_m]:X\to \PP^m$, and write
$$F=\Phi\circ f: \CC \to \PP^m.$$
When $E_1,\ldots, E_\ell$ denote the irreducible components of $D$, let
$$M=\max\{-\ord_{E_i}(\phi_j)\ |i=1,\ldots,\ell,\ j=0,\ldots,m\}.$$
Then, for every $\epsilon>0$,
$$T_F(r) .\leq. (M(m+1)+1+\epsilon) N_f(r,D),$$
where ``$.\leq.$'' means that the inequality holds for all $r$ outside of a set $\Gamma$ with finite Lebesgue measure.
\end{theorem}
Note that since $\Phi$ is nonconstant and $f$ is algebraically nondegenerate, the map $F$ also is nonconstant.\par

\begin{corollary} \label{hyp_cor} Let $D$ be an essentially large divisor on a
  nonsingular complex projective variety $X$. Then every holomorphic
  map $f: \CC \to X \setminus D$ must be algebraically degenerate,
  i.e., the image of $f$ must be contained in a proper subvariety of
  $X$. In other words, $X \setminus D$ is Chern hyperbolic (aka quasi Brody hyperbolic).
\end{corollary}
\begin{proof}[Proof of Corollary \ref{hyp_cor}] Since the statement of the Corollary only refers to the support of $D$, we can assume that $D$ is essentially {\it very} large. Assume that 
$f: \CC \to X\setminus D$ is algebraically nondegenerate.  Since $f$ omits $D$,  $N_f(r,D)=0$. So the above inequality implies that $T_F(r)$ is bounded outside of a set of finite Lebesgue measure. However, since $F$ is nonconstant, this is false and gives a contradiction.
\end{proof}
\begin{proof}[Proof of Theorem \ref{our_smt}]
It is easy to see (\cite[Remark 8.2]{levin_annals}) that there exists a finite set $J$ of elements in $V$ such that for every $P\in D$ there exists a subset $I\subset J$ that is a basis of $V$ with $\ord_{E}\prod_{f\in I} f>0$ when $E$ is a component of $D$ with $P\in E$. Let $J'$ be the set of linear forms $L$ in $m+1$ variables such that $L\circ\Phi\in J$.\par
According to Theorem \ref{gen_smt},
$$\int_0^{2\pi}\max_I \log \prod_{L\in I} \frac{\|F(r e^{\sqrt{-1}\theta})\|\|L\|}{|L(F(re^{\sqrt{-1}\theta}))|}\frac{d\theta}{2\pi}.\leq. (m+1+\varepsilon)T_F(r),$$
where the maximum is taken over subsets $I\subset J'$ such that $I$ consists of exactly $m+1$ independent linear forms.\par
Since the left-hand side is independent of the choice of the representation of $F$, we can rewrite the above inequality as 
$$\int_0^{2\pi}\max_I \log \prod_{L\in I} \frac{\max_{0 \leq j \leq
    m}|\phi_j(f(re^{\sqrt{-1}\theta}))|}
{|L(F(re^{\sqrt{-1}\theta}))|}\frac{d\theta}{2\pi}.\leq.
(m+1+\varepsilon)T_F(r) +O(1),$$
where $O(1)$ is introduced by dropping the factor $\|L\|$ after assuming that a fixed choice of coefficients has been made.\par
Next, we observe that since $D$ is essentially very large, outside of the support of the divisor $D$, we have
\begin{equation}\label{ess_v_l_boundedness}
\max_j |\phi_j(f(re^{\sqrt{-1}\theta}))|^{\frac 1 M} \min_I \prod_{L\in I} {|L(F(re^{\sqrt{-1}\theta}))|} \leq  C,
\end{equation}
where $C$ is some constant. The reason is that we can cover $X$ by a finite number of open sets on each of which $|\phi_j\min_I (\prod_{L\in I} {|L\circ \Phi |)^M}|$ is bounded for each $j=0,\ldots,m$. Namely, for $P\in X$, define $U_P$ as follows. If $P\not \in D$, take a small neighborhood $U_P$ of $P$ such that its closure is disjoint from the support of $D$. Since the $\phi_j$ have no poles outside of the support of $D$, the boundedness on $U_P$ is obvious. If $P\in D$, take $U_P$ to be a small neighborhood of $P$ with the property that $Q\in E\cap U_P$ implies $P\in E$ for all components $E$ of $D$. Consequently, by the essential very largeness, there exists $I_0$ such that $ \ord_E (\prod_{L\in I_0} {L\circ \Phi )}$ is a positive integer and $\ord_E\left(\phi_j\min_I (\prod_{L\in I} {|L\circ \Phi |)^M}\right)\geq 0$ for all $j=0,\ldots,m$. Perhaps after shrinking $U_P$, this implies that $|\phi_j\min_I (\prod_{L\in I} {|L\circ \Phi |)^M|}$ is bounded on $U_P$ for every $j=0,\ldots,m$. Finally, since $X$ is compact, we can cover it with a finite number of open sets of the form $U_P$.\par
Applying $\log$ to both sides of \eqref{ess_v_l_boundedness} yields
$$\log \max_j |\phi_j(f(re^{\sqrt{-1}\theta}))|^{\frac 1 M}
+  \log \min_I \prod_{L\in I} {|L(F(re^{\sqrt{-1}\theta}))|} \leq  \log C.$$
Therefore, 
$$\log \max_j
|\phi_j(f(re^{\sqrt{-1}\theta}))|^{m+1+\frac 1 M}
\leq \max_I \log \prod_{L\in I}\frac{\max_j |\phi_j(f(re^{\sqrt{-1}\theta}))|} 
{{|L(F(re^{\sqrt{-1}\theta}))|}} +\log C.$$
Thus
 \begin{gather*}
\int_0^{2\pi} \log \max_j |\phi_j(f(re^{\sqrt{-1}\theta}))| d\theta 
 \\ = \frac {M}{M(m+1)+1} 
\int_0^{2\pi}\log  \max_j |\phi_j(f(re^{\sqrt{-1}\theta}))|^{m+1+\frac{1}{M}} d\theta \\
\leq  \frac {M}{M(m+1)+1}  \left(  \int_0^{2\pi}
\max_I \log {\prod_{L\in I} \frac{\max_j |\phi_j(f(re^{\sqrt{-1}\theta}))|}  {|L(F(re^{\sqrt{-1}\theta}))|}}  d\theta +\log C \right) \\
.\leq.   \frac {M(m+1+\epsilon)}{M(m+1)+1} 
T_F(r)  +O(1).
\end{gather*}
Recall that from Lang's formula \cite[p. 202]{lang_book},
\begin{align*}
T_F(r)=&\int_0^{2\pi} \log
\max_j|\phi_j(f)(re^{\sqrt{-1}\theta})|\frac{d\theta}{2\pi}-
\log\max_{j\in K} |c_{\phi_j(f)}|\\
&+\sum_{a\in D_r\setminus{0}}\max_{j=0,\ldots,m}(-\ord_a(\phi_j(f)))\log\left|\frac r a \right|\\ &+\max_{j=0,\ldots,m}(-\ord_0(\phi_j(f)))\log |r|.
\end{align*}
Thus, 
\begin{align*}
T_F(r)  .\leq.&  \frac{M(m+1+\epsilon)}{M(m+1)+1} T_F(r)
+ \sum_{a\in
  D_r\setminus{0}}\max_{j=0,\ldots,m}(-\ord_a(\phi_j(f)))\log\left|\frac
  r a \right|\\&+\max_{j=0,\ldots,m}(-\ord_0(\phi_j(f)))\log |r|  +O(1).
\end{align*}
Recall that due the isomorphism \eqref{h0_isom}, there are sections $s_0,\ldots , s_m\in H^0(X,\OO_X(D))$ such that $\phi_j=\frac{s_j}{s_D}$. Therefore, for all $a\in D_r$,
\begin{align*}
\max_{j} (-\ord_a(F_j))&=\max_{j}(-\ord_a(\phi_j\circ f))\\
&=\max_{j}\left(-\ord_a\left(\frac{s_j}{s_D}\circ
    f\right)\right)\\&=\ord_a(s_D\circ
f)-\min_{j}\ord_a(s_j\circ f) \\&\leq \ord_a(s_D\circ f).
\end{align*}
Consequently, we obtain
\begin{align*}  \frac {1-\epsilon}{M(m+1)+1} T_F(r)
.\leq. &\sum_{a\in D_r\setminus\{0\}}\ord_a(s_D\circ f)\log\left|\frac r a \right|\\ &+\ord_0(s_D\circ f)\log r+O(1)\\
=&N_f(r,D) +O(1).
\end{align*}\par
Since $T_F(r)$ is unbounded for $r\to \infty$ outside of a set of finite Lebesgue measure, we can, after perhaps enlarging $\Gamma$, continue to estimate as follows, where $\delta$ is any positive number.
\begin{align*}  &N_f(r,D) +O(1)\\
.\leq.& N_f(r,D)+\delta T_F(r).
\end{align*}
We can now subtract $\delta T_F(r)$ from both sides of the inequality and then divide both sides by $\frac {1-\epsilon}{M(m+1)+1}-\delta$. Choosing small enough $\epsilon$ and $\delta$ (depending on $\epsilon$) concludes the proof.
\end{proof}
\subsection{The arithmetic case}\label{arithm_case}
We now derive an inequality of Schmidt Subspace Theorem-type for essentially large divisors. We again begin by recalling some standard definitions.\par
Let $k$ be a number field of degree $d$. Denote by $M_k$ the set
of places (equivalence classes of absolute values) of $k$ and
write $M^{\infty}_k$ for the set of 
archimedean places of $k$.  In an archimedean class $\upsilon$, we choose the absolute value $|~~|_{\upsilon}$ such that 
$|~~|_{\upsilon}
= |~~|$ on $\QQ$ (the standard absolute value). For a nonarchimedean class $\upsilon\in M_k\setminus M^{\infty}_k$, we let $|p|_{\upsilon} = p^{-1}$ if $\upsilon$ lies above the 
rational 
prime $p$. \par
Denote by $k_{\upsilon}$ 
the completion of $k$ with respect to $\upsilon$ and by 
$d_{\upsilon} = [k_{\upsilon}
: {\QQ}_{\upsilon}]$ the local degree. Let $||~~||_{\upsilon} = 
|~~|_{\upsilon}^{d_{\upsilon}/d}.$ 
We have the product formula: for every $x\in k^*$, 
$$\prod_{\upsilon\in M_k} \|x\|_{\upsilon}=1.$$
 For $Q=[x_0, \dots, x_m] \in \PP_k^{m}$, let 
 $\|Q\|_{\upsilon} = 
\max _{1 \leq i \leq m}\|x_i\|_{\upsilon}$. Moreover, define the  logarithmic  height of
$Q$ by 
$$h(Q) =\sum_{\upsilon \in {M_k}} \log \|{Q}\|_{\upsilon}.$$
By the product formula, its definition is independent of the choice of the 
representations. \par
We now come to the notion of a Weil function. Let $X$ be a projective nonsingular variety and $D$ be an effective divisor on $X$, both defined over 
a number field $k$. Extend $\|~~\|_{\upsilon}$ to an 
absolute value on the algebraic closure $\bar{k}_{\upsilon}$ for all $\upsilon\in M_k$. Then a 
{\it (local) Weil function for $D$ relative to $\upsilon$} is a function
$\lambda_{D, \upsilon}: X(\bar{k}_{\upsilon})\setminus  \mbox{Supp}(D)\rightarrow {\RR}$ 
such that if $D$ is represented locally by $(f)$
on an open set $U$, then 
$$\lambda_{D, \upsilon}(P): =-\log\|f(P)\|_{\upsilon}+\alpha(P),$$
where $\alpha(P)$ is a continuous function on $U({\bar k}_{\upsilon}).$
We sometimes think of $\lambda_{D, \upsilon}$ as a function of 
$X(k)\setminus \mbox{Supp}(D)$ or $X(\bar{k})\setminus \mbox{Supp}(D)$ by simply choosing an embedding 
${\bar k}\rightarrow \bar{k}_{\upsilon}.$\par
Given a finite set $S \subset M_k$ containing all archimedean places, define the {\it proximity function} 
$m_S(x, D)$ by, for $x \in X({\bar k})\setminus \mbox{Supp}(D) $,
\begin{equation}m_S(x, D) = {1\over [k:{\QQ}]} \sum_{\upsilon \in S}  \lambda_{D, \upsilon}(x).
\end{equation}
The {\it counting function} $N_S(x, D)$ is defined by 
\begin{equation}N_S(x, D) = {1\over [k:{\QQ}]} \sum_{\upsilon \not\in S}   \lambda_{D, \upsilon}(x).
\end{equation}
Note that the sum above is still a finite sum, since the terms all vanish except for finitely 
many.  

\begin{definition} \label{int_pt_def} A subset $\Sigma$ of $X({\bar k})\setminus \mbox{Supp}(D)$ is 
$(D, S)$-{\it integral} if there is a Weil function 
$\lambda_D$ for $D$ and constants $c_{\upsilon}\in {\RR}$ for all 
$\upsilon\not\in S$ such that $c_{\upsilon}=0$ for almost all $c_{\upsilon}=0$
 and $\lambda_{D, \upsilon}(P)\leq c_{\upsilon}$ for all  
$\upsilon\not\in S$ and $P\in \Sigma$. In other words, a subset $\Sigma$ of $X({\bar k})\setminus \mbox{Supp}(D)$ is 
$(D, S)$-{\it integral} if and only if $N_S(x, D)$ is bounded over $\Sigma$.
\end{definition}\par

We recall the following generalized version of Schmidt's Subspace Theorem from \cite{vojta_ajm_87}.
\begin{theorem}  \label{schmidt} Let $k$ be a number field 
with its set of canonical places $M_k$. Let $S \subset M_k$ be a finite set containing 
all archimedean places.  Let $H_1, \dots, H_m$ be hyperplanes in 
$\PP^n$ defined over ${\bar k}$ with corresponding 
Weil functions $\lambda_{H_1}, \dots, \lambda_{H_m}$. 
Then there exists a finite union of hyperplanes $Z$, depending only
on $H_1, \dots, H_m$ (and not $k, S$), such that for 
any $\epsilon>0$, 
$$\sum_{\upsilon \in S} \max_I \sum_{i\in I} 
\lambda_{H_i, \upsilon}(P) \leq (n+1+\epsilon)h(P)$$
holds for all but finitely many $P\in {\PP}^n\setminus Z$, where the maximum is taken over subsets $I \subset \{1, \dots, m\}$ such that the linear forms defining $H_i, i\in I,$ are linearly independent.
\end{theorem}\par
Our main result in this Subsection is the following inequality of Schmidt Subspace Theorem-type for essentially large divisors. 
\begin{theorem} \label{our_schmidt} Let $D$ be an essentially very large divisor on a nonsingular  projective variety $X$, both defined over a number field $k$.  Let $V$ be a linear subspace of $L(D)$ as in Definition \ref{ess_v_l_def}. Let $\phi_0,\ldots\phi_m$ be an arbitrary basis for $V$ and let $\Phi$ be the rational map $[\phi_0:\ldots:\phi_m]:X\to \PP^m$. Let $S\subset M_k$ be a finite set of places containing all archimedean ones. Then for every $\epsilon >0$,
$$h(\Phi(P)).\leq.  (M(m+1)+1+\epsilon)N_S(P, D),$$
where ``$.\leq.$'' means that the inequality holds for all $P\in X(k)$ outside a Zariski closed
variety $Z$ of $X$.
\end{theorem}
\begin{proof}
Let $I$ be defined as in the proof of Theorem \ref{our_smt}. Due to arguments analogous to those that gave the boundedness \eqref{ess_v_l_boundedness}, we have for $P\in X(k)$ not contained in the support of $D$:
$$ \sum_{\upsilon \in S} \max_j
\log \|\phi_j(P)\|_{\upsilon}^{m+1+\frac 1 M}
\leq \sum_{\upsilon \in S} \max_I \log \prod_{L\in I}\frac{\max_j 
 \|\phi_j(P)\|_{\upsilon}} {\|L(\Phi(P))\|_{\upsilon}} +\log C,$$
for some constant $C$. In the sequel, we will denote abuse and denote all constants with the same symbol $C$.\par
By Theorem \ref{schmidt},  for all $P\in X(k)$ outside a Zariski closed
variety  $Z$ of $X$,
$$\sum_{\upsilon \in S} \max_I \log \prod_{L\in I}\frac{\max_j 
 \|\phi_j(P))\|_{\upsilon}} {\|L(\Phi(P))\|_{\upsilon}} 
.\leq.  (m+1+\epsilon)
h(\Phi(P)).$$
Hence
$$
\sum_{\upsilon \in S} \max_j
\log \|\phi_j(P)\|_{\upsilon}^{m+1+\frac 1 M} .\leq. (m+1+\epsilon)
h(\Phi(P))+\log C.$$
Now, 
$$h(\Phi(P)) =\sum_{\upsilon \in S} \max_j
\log \|\phi_j(P)\|_{\upsilon}+ \sum_{\upsilon\not \in S} \max_j
\log \|\phi_j(P)\|_{\upsilon}.$$\par
We again use the fact that, due to the isomorphism \eqref{h0_isom}, there are sections 
$s_0,\ldots , s_m\in H^0(X,\OO_X(D))$ such that $\phi_j=\frac{s_j}{s_D}$, $j=0,\ldots,m$, where $s_D$ is a section of $\OO_X(D)$ with $\divisor (s) =D$. Moreover, we think of a section $s$ of $\OO_X(D)$ as given by a collection $(s_\beta)$ of functions pertaining to an open cover $(U_\beta)$ such that $s_\beta=h_{\beta\gamma}s_\gamma$ on $U_\beta\cap U_\gamma$, where $h_{\beta\gamma}$ are the nowhere vanishing transition functions of $\OO_X(D)$.\par
If $P\in U_\beta(\bar{k}_{\upsilon})$, then we can write 
$$\|\phi_j(P)\|_{\upsilon}=\frac{\|s_{j,\beta}(P)\|_{\upsilon}}{\|s_{D,\beta}(P)\|_{\upsilon}}$$
independently of $\beta$.
Thus
$$
 \sum_{\upsilon \in S} \max_j
\log \|\phi_j(P)\|_{\upsilon} - \sum_{\upsilon\not \in S} \log {\|s_{D,\beta}(P)\|_{\upsilon}
\over   \max_j
\|s_{j,\beta}(P)\|_{\upsilon}}
= h(\Phi(P)).$$
From the definition of the Weil-function $ \lambda_{D, \upsilon} (P)$,
since $D$ is represented by $s_{D, \beta}$ on $U_\beta$, 
$$\lambda_{D, \upsilon} (P) =-\log \|s_{D, \beta}(P)\|_{\upsilon}+\alpha(P),$$ 
where $\alpha$ is continuous on $U_\beta(\bar{k}_{\upsilon})$.
Since $\log  \max_j \|s_{j, \beta}(P)\|_{\upsilon}$ is bounded from above on $U_\beta(\bar{k}_{\upsilon})$, there exists a constant $C_\beta$ such that
\begin{align*}
- \log {\|s_{D, \beta}(P)\|_{\upsilon}
\over   \max_j
\|s_{j, \beta}(P)\|_{\upsilon}}
=&-\log \|s_{D, \beta}(P)\|_{\upsilon} + \log  \max_j \|s_{j, \beta}(P)\|_{\upsilon}\\
\leq&  -\log \|s_{D, \beta}(P)\|_{\upsilon}+\alpha(P)+C_\beta\\=&\lambda_{D, \upsilon} (P) +C_\beta .
\end{align*}
Therefore
$$\sum_{\upsilon \in S} \max_j
\log \|\phi_j(P)\|_{\upsilon} +  \sum_{\upsilon\not \in S}  \lambda_{D, \upsilon} (P) +C
\geq  h(\Phi(P)).$$
So we have
 \begin{align*}
 \left(m+1+{1\over M}\right) h(\Phi(P))
 \leq&
\sum_{\upsilon \in S} \max_j
\log \|\phi_j(P)\|_{\upsilon}^{m+1+\frac 1 M} \\
&+\left(m+1+{1\over M}\right)\sum_{\upsilon\not \in S} \lambda_{D, \upsilon} (P) +C\\
.\leq. & (m+1+\epsilon) 
h(\Phi(P))+ \left(m+1+{1\over M}\right)
\sum_{\upsilon\not \in S} \lambda_{D, \upsilon} (P) \\
&+C.
\end{align*}
By Northcott's Theorem, for any $\delta>0$, $\{Q\in \PP^m | h(Q)< \frac C \delta \}$ is a finite set. Thus, after possibly enlarging $Z$, we can continue to estimate
\begin{align*}
 &(m+1+\epsilon)  h(\Phi(P))+ \left(m+1+{1\over M}\right)
\sum_{\upsilon\not \in S} \lambda_{D, \upsilon} (P) +C\\
.\leq.& (m+1+\epsilon+\delta)  h(\Phi(P))+ \left(m+1+{1\over M}\right)
\sum_{\upsilon\not \in S} \lambda_{D, \upsilon} (P).
\end{align*}
After rearranging the inequality and choosing $\epsilon$ and $\delta$ (depending on $\epsilon$) small, we obtain the Theorem.
\end{proof}
In strict analogy to Corollary \ref{hyp_cor}, we obtain the following
\begin{corollary} Let $D$ be an essentially large divisor on a nonsingular projective variety $X$, both defined over a number field $k$. Let $S\subset M_k$ be a finite set of places containing all archimedean ones. Then any set of $(D,S)$-integral points is contained in a proper subvariety of $X$. In other words, $X\setminus D$ is quasi Mordellic.
 \end{corollary}
\begin{proof}
We can assume without loss of generality that $D$ is essentially {\it very} large. Let $\Sigma$ be a set of $(D,S)$-integral points. According to Definition \ref{int_pt_def}, $N_S(x,D)$ is bounded over $\Sigma$. So, by Theorem \ref{our_schmidt}, $h(\Phi(P))$ is bounded for all $P\in \Sigma$ outside a Zariski closed variety $Z$ of $X$. This means that, due to Northcott's Theorem, the set $\{\Phi(P)|P\in (\Sigma \setminus Z)(k)\}$ is finite. Since $\Phi$ is a nonconstant morphism outside of $D$, the preimage of this finite set is a proper subvariety $Z_1$ of $X$. We have shown that $\Sigma\subset Z\cup Z_1$, which concludes the proof.
 \end{proof}
\section{Proofs of the sharp sufficient criteria for essentially large divisors}\label{section_ess_v_l_proofs}
As was stated in the Introduction, the filtration method from \cite{cz_ajm}, \cite{ru_ajm} can be used to prove the sharp bound for the number of components of a large divisor in the case of $X=\PP^q$ (Theorem \ref{thm_1}). We now give the proof. 
\begin{proof}[Proof of Theorem \ref{thm_1}]
The number field case being completely analogous (see \cite{cz_ajm}), we only prove the case when $K=\CC$.\par
Let $Q_1,\ldots,Q_r$ be homogeneous polynomials of degree $d_i$ such that $D_i=\{Q_i=0\}$. Let $P\in D$. After replacing each $Q_i$ by $Q_i^{d/d_i}$, where $d$ is the least common multiple of the $d_i$'s, we can assume that all $Q_i$ have the same degree $d$.  Without loss of generality, assume that precisely $q$ of the $D_i$ contain the point $P$, and let $\{\gamma_1,\ldots,\gamma_q\}$ be the unique set of distinct elements of $\{Q_1,\ldots,Q_r\}$ such that $P\in \cap_{i=1}^q\{\gamma_i=0\}$.\par
Let $V_N$ be the vector space of homogeneous polynomials of degree $N$ in $\CC[X_0,\ldots,X_q]$. With the lexicographical ordering on the $q$-tuples $\vec{i}=(i_1,\ldots,i_q)$ with $\sigma(\vec{i}):=\sum_{j=1}^q i_j\leq N/d$, we obtain a filtration on $V_N$ by letting
$$W_{\vec{i}}=\sum_{\vec{e}\geq \vec{i}}\gamma_1^{e_1}\ldots\gamma_q^{e_q}V_{N-d\sigma(\vec{e})}.$$ 
Note that clearly $W_{\vec{0}}=V_N$ and $W_{\vec{i}}\supseteq W_{\vec{i}'}$ for $\vec{i'}\geq \vec{i}$.
Combining Lemma 2.3 and Lemma 3.1 of \cite{cz_ajm}, we note the following Lemma for our situation.
\begin{lemma}\label{hilbert_lemma}
There exists an integer $N_0$ (depending only on $\gamma_1,\ldots,\gamma_q$) such that for all integers $N> N_0$ and for all $\vec{i}$ with $d\sigma(\vec{i}) < N-N_0$:
$$\Delta_{\vec{i}}:=\dim W_{\vec{i}}/W_{\vec{i'}}=d^q.$$
\end{lemma}
Now choose a basis $\psi_1,\ldots,\psi_m$ ($m={N+q\choose q}$) for $V_N$ with respect to the above filtration. For all $\nu=1,\ldots,m$, write
$$\psi_\nu=\gamma_1^{i_1}\ldots\gamma_q^{i_q}\gamma^{(\nu)}$$ 
for some $\gamma^{(\nu)}\in V_{N-d\sigma(\vec{i})}$ according to its place in the filtration.\par
For $\tilde N\in \NN$, let $N=\tilde N\cdot r \cdot d$ and write
$$H^0(\PP^q,\OO_{\PP^q}(\tilde N D))=H^0(\PP^q,\OO_{\PP^q}(\tilde N\cdot r \cdot d))=H^0(\PP^q,\OO_{\PP^q}(N))=V_N.$$
For $\nu=1,\ldots,m$, let
$$f_\nu=\frac{\psi_\nu}{Q_1^{\tilde N}\ldots Q_r^{\tilde N}}.$$ 
These rational functions clearly form a basis for $L(\tilde N D)$. To conclude the proof of the Theorem, we show that they satisfy
 $\ord_{E} \prod_{\nu=1}^{m} f_\nu > 0$ for any component $E$ in $D$ with $P\in E$. We assume that $E$ is contained in $D_{j_0}$. Then
$$
\frac 1 {\ord_E D} \ord_{E} \prod_{\nu=1}^{m} f_\nu = \left(\sum_{\vec{i}}\Delta_{\vec{i}}i_{j_0}\right)-\tilde N {N+q\choose q}.$$
Since the number of nonnegative integer $m$-tuples with 
sum $\leq t$ is equal to the number of nonnegative integer 
$(m+1)$-tuples with sum exactly $t$ which is 
${t+m\choose m}$, and since the sum below is independent of 
$j$, we have that, for 
$N$ divisible by $d$ and for every $j$, 
\begin{align*}
\sum_{\vec{i}} i_j&={1\over q+1}\sum_{\hat{\vec{i}}} \sum_{j=1}^{q+1}
i_j ={1\over q+1} \sum_{\hat{\vec{i}}}  {N\over d}\\
&={1\over q+1} {N/d+q\choose q}{N\over d}
={N^{q+1}\over d^{q+1}(q+1)!}+O(N^q),
\end{align*}
where $\displaystyle{\sum_{\hat{\vec{i}}}}$ is taken over all nonnegative integer $(q+1)$-tuples with sum exactly $N/d$. 
Combining it with Lemma \ref{hilbert_lemma},
we have, for every $1 \leq j_0 \leq q$, 
$$\sum_{\vec{i}}\Delta_{\vec{i}}i_{j_0} ={N^{q+1}\over d(q+1)!}+O(N^q).$$
Thus, 
\begin{align*}
\frac 1 {\ord_E D} \ord_{E} \prod_{\nu=1}^{m} f_\nu &= \left(\sum_{\vec{i}}\Delta_{\vec{i}}i_{j_0}\right)
-\tilde N {N+q\choose q}\\&=\left(\frac{N^{q+1}}{d\cdot (q+1)!}+O(N^q)\right)- \frac N{dr} \left(\frac {N^q} {q!}+O(N^{q-1})\right)\\
&=\left(\frac{1}{q+1}-\frac{1}{r}\right) \frac{N^{q+1}}{d\cdot q!} + O(N^q).
\end{align*}
For $N$ sufficiently large, this is positive if $r \geq q+2$.
\end{proof}\par
\begin{remark} We remark that under the assumptions of Theorem \ref{thm_1}, the paper \cite{ru_ajm}, using the same methods from \cite{cz_ajm}, obtained a sharp Second Main Theorem that is stronger than Theorem \ref{our_smt}. The same can be said for Theorem \ref{thm_2} and the paper \cite{ru_annals}, which is in turn based on \cite{ferretti_compositio}, \cite{ef_imrn}. On the other hand, Theorem \ref{our_smt} does not require these assumptions, so it is still a new result.\end{remark}

We continue with the proof of Theorem \ref{thm_2}, based on the methods of \cite{ferretti_compositio}, \cite{ef_imrn} and  \cite{ru_annals}.
\begin{proof}[Proof of Theorem \ref{thm_2}]
The number field case being completely analogous (see \cite{ferretti_compositio}), we only prove the case when $K=\CC$.\par
Let $Q_1,\ldots,Q_r$ be homogeneous polynomials such that $D_i=\{Q_i=0\}\cap X$.
As above, we can assume that the degrees of the homogeneous polynomials $Q_i$ equal $d$ for all $i$. Let
$$\varphi: X\to \PP^{r-1}, \quad x\mapsto [Q_1(x),\ldots,Q_r(x)].$$
Let $Y:=\varphi(X)$. By the general position assumption, $\varphi$ is a finite holomorphic map $X\to Y$. \par
On $\PP^{r-1}$, we have for all $N\in \NN$ a short exact sequence 
$$0\to \II_Y(N)\to \OO_{\PP^{r-1}}(N)\to \OO_Y(N)\to 0.$$
The beginning of the corresponding long exact sequence reads
$$0\to H^0(\PP^{r-1},\II_Y(N))\to H^0(\PP^{r-1},\OO_{\PP^{r-1}}(N))\stackrel {\tau}{\to} H^0(Y,\OO_Y(N)),$$
where $\tau$ denotes the restriction map. We let 
\begin{align*}
W_N&:=\image(\tau)\\& \cong H^0(\PP^{r-1},\OO_{\PP^{r-1}}(N))/ \ker(\tau)\\ &\cong H^0(\PP^{r-1},\OO_{\PP^{r-1}}(N))/H^0(\PP^{r-1},\II_Y(N))\\ &\cong\CC[Y_0,\ldots,Y_{r-1}]_N/(I_Y)_N,
\end{align*}
where $(I_Y)_N$ denotes the set of those homogeneous polynomials of degree $N$ vanishing on $Y$. 
Let $V_N:=\varphi^*(W_N)\subseteq H^0(X,\OO_X(ND))$.\par
Since $\varphi:X\to Y$ is a finite surjective holomorphic map,
\begin{align*}
\dim(V_N)=\dim (W_N)=\dim \CC[Y_0,\ldots,Y_{r-1}]_N/(I_Y)_N=H_Y(N),
\end{align*}
where $H_Y(N)$ is the Hilbert function of $Y$.\par
Let $P\in D$. Without loss of generality, we assume again that there are distinct 
$Q_{i_1},\ldots,Q_{i_q} \in \{Q_1,\ldots,Q_r\}$ such that $P\in \cap_{j=1}^q\{Q_{i_j}=0\}$.\par
Let $\vec{c}=(c_1, \dots, c_r)$ be the 
$r$-vector whose $i_j$-th entry ($1 \leq j \leq q$) is $1$ and 
elsewhere is $0$.
Let $\vec{y}^{\vec{a}^{(1)}},\ldots,\vec{y}^{\vec{a}^{(H_Y(N))}}$ be monomials such that their equivalence classes in $\CC[Y_0,\ldots,Y_{r-1}]_N/(I_Y)_N$ give a basis and such that
$$S_Y(N,\vec{c})=\sum_{i=1}^{H_Y(N)} \vec{a}^{(i)} \bullet \vec{c},$$
where $S_Y(N,\vec{c})$ is the $N$-th Hilbert weight and the bullet denotes the usual dot product. Recall that the  $N$-th Hilbert weight
 is given by 
$$S_Y(N,\vec{c})=\max \sum_{i=1}^{H_Y(N)} \vec{a}^{(i)} \bullet \vec{c},$$
where the maximum is taken over  all sets of monomials
  $\vec{y}^{\vec{a}^{(1)}},\ldots,\vec{y}^{\vec{a}^{(H_Y(N))}}$ whose residue class modulo
$I_Y$ form a basis of $\CC[Y_0,\ldots,Y_{r-1}]_N/(I_Y)_N$.
For $\nu=1,\ldots,H_Y(N)$, and $N$ a multiple of $r$, let 
$$f_\nu = \frac{Q_1^{a_1^{(\nu)}}\ldots Q_r^{a_r^{(\nu)}}}{Q_1^{N/r}\ldots Q_r^{N/r}}|_X.$$
These functions form a basis for $V_N$ understood as a subspace of $L(ND)$. To conclude the proof of the Theorem, we show that they satisfy
 $\ord_{E} \prod_{\nu=1}^{H_Y(N)} f_\nu > 0$ for any component $E$ in
 $D$ with $P\in E$. We assume that $E$ is contained in $D_{j_0}$. \par
We recall two basic lemmas in \cite{ru_annals}:
\begin{lemma} Let $X \subset {\PP}^{N}_{\CC}$ 
be an algebraic variety of 
dimension $n$ and degree $\bigtriangleup$. Let $m>\bigtriangleup$ 
be an integer and
let   ${\vec c}=(c_0, \dots, c_N) \in {\Bbb R}^{N+1}_{\ge 0}$.
Then
$$ {1\over mH_X(m)}S_X(m, {\vec c})
\ge {1\over (n+1)\bigtriangleup} e_X({\vec c}) - {(2n+1)\bigtriangleup
\over m}\cdot
\left(\max_{i=0, \dots, N}c_i\right),$$
where ${1\over (n+1)\bigtriangleup} e_X({\vec c})$ is the normalized Chow weight of $X$ with respect to $\vec{c}$.
\end{lemma}
We will not give a definition of the normalized Chow weight here (see \cite{ru_annals}). For our purposes, all we need is the following estimate from below.
 \begin{lemma} 
Let $Y$ be a subvariety of ${\Bbb P}^{q-1}_{\CC}$ of dimension $n$ and degree 
$\bigtriangleup$. Let ${\vec c}=(c_1, \dots, c_q)$ be a tuple of
positive reals. 
Let $\{i_0, \dots, i_n\}$ be a subset of $\{1, \dots, q\}$ such that 
$$Y\cap \{y_{i_0}=0, \dots, y_{i_n}=0\}=\emptyset. $$
Then 
$$e_Y({\vec c})\ge (c_{i_0}+\cdots +c_{i_n})\cdot \bigtriangleup. $$
\end{lemma}
We now continue our proof. With our chosen $\vec{c}$ and $\vec{a}^{(i)}$, using Lemmas 3.3 
and 3.4 (notice the condition that $Q_1, \dots, Q_r$ are in general position), and the symmetry property of the $\vec{a}^{(1)},\ldots,\vec{a}^{(H_Y(N))}$ (\cite[p. 61]{mumford_enseign}),
\begin{align*}
{1\over \ord_E D}\ord_E \prod_{\nu=1}^{H_Y(N)} f_\nu &= \left(\sum_{\nu=1}^{H_Y(N)} a_{j_0}^{(\nu)}\right)-\frac{N}{r}H_Y(N)\\
&=\frac 1 q \left(\sum_{\nu=1}^{H_Y(N)}  \vec{a}^{(\nu)} \bullet \vec{c}\right)- \frac{N}{r}H_Y(N)\\
&=\frac 1 q S_Y(N,\vec{c})- \frac{N}{r}H_Y(N)\\
&\ge\frac 1 q \frac 1 {q+1} N H_Y(N)(\sum_{j=1}^q  c_{i_j})-O(H_Y(N))- \frac{N}{r}H_Y(N)\\
&=\frac 1 {q+1} N H_Y(N) - \frac{N}{r}H_Y(N)-O(H_Y(N))\\
&=(\frac 1 {q+1} - \frac{1}{r})NH_Y(N)- O(H_Y(N)).\\
\end{align*}
For $N$ sufficiently large, this is positive if $r \geq q+2$.
\end{proof}

\section{Essentially large divisors on projective varieties of small codimension}\label{section_alg_geom}
Our ultimate (and so far unreached) goal is of course to find a (sharp) lower bound for the number of (say ample) components in general position necessary for a divisor on a nonsingular projective variety to be essentially large (cf.\ Conjecture \ref{conjecture}). For a general nonsingular projective variety, it is clearly not true that all effective divisors are cut out by hypersurfaces in the ambient projective space, as assumed in Theorem \ref{thm_2}. In this final Section, we exhibit a class of nonsingular projective varieties  for which this is the case and to which Theorem \ref{thm_2} can thus be applied. Consequently, the above-mentioned ultimate goal has been achieved for these varieties.\par
We refer the reader to \cite{hartshorne_bull_ams} or \cite[Section 3.2]{PAGI} for more on the algebraic geometry behind this question, which is essentially the problem of extending the Lefschetz Hyperplane Theorem to varieties that are not complete intersections. In particular, the theorems of Barth \cite{barth_ajm}, Larsen \cite{larsen_invent}, and, in the case of a general ground field of characteristic zero, Ogus \cite{ogus_annals} apply. For our purposes, we simply state the following Proposition, which is an immediate consequence of these theorems.
\begin{proposition} Let $X\subset \PP^\ell$ be a nonsingular projective variety of dimension $q$, defined over $K$. If $2q - \ell \geq 2$, then restriction yields an isomorphism
$$\Pic(\PP^\ell)\stackrel{\cong}{\to} \Pic(X).$$
\end{proposition}\par
Since $\Pic(\PP^\ell)=\ZZ$, this Proposition gives the following Corollary to Theorem \ref{thm_2}.
\begin{corollary}\label{cor}
Let $q\geq 1$ and $r\geq q+2$ be integers. Let $X\subseteq \PP^\ell$ be a nonsingular projective variety of dimension $q$, defined over $K$. Assume that $2q - \ell \geq 2$ holds. Let $D=\sum_{i=1}^r D_i$ be an effective divisor on $X$ defined over $K$ such that the $D_i$ are in general position. Then $D$ is essentially large.
\end{corollary}
We conclude by pointing out that the condition $2q - \ell \geq 2$ is not that restrictive. It is well-known that, by general linear projections, any nonsingular projective variety of dimension $q$ can be embedded in $\PP^\ell$ with $\ell=2q+1$. Thus, for a generic nonsingular projective variety, one can always find an embedding with $2q-\ell=-1$, which is quite close to the above condition. In particular, Corollary \ref{cor} applies to many interesting special projective varieties such as hypersurfaces, appropriate complete intersections, and certain Grassmannians, such as $G(2,5)$ embedded into $\PP^{9}$ under the Pl\"ucker embedding. Note that the latter is not a complete intersection due to B\'ezout's Theorem, because its degree is (the prime number) $5$, while it is not contained in any hyperplane.


\begin{thebibliography}{Mum77}

\bibitem[Bar70]{barth_ajm}
W.~Barth.
\newblock Transplanting cohomology classes in complex-projective space.
\newblock {\em Amer. J. Math.}, 92:951--967, 1970.

\bibitem[CZ02]{cz_cr}
P.~Corvaja and U.~Zannier.
\newblock A subspace theorem approach to integral points on curves.
\newblock {\em C. R. Math. Acad. Sci. Paris}, 334(4):267--271, 2002.

\bibitem[CZ03]{cz_crelle}
P.~Corvaja and U.~Zannier.
\newblock On the number of integral points on algebraic curves.
\newblock {\em J. Reine Angew. Math.}, 565:27--42, 2003.

\bibitem[CZ04a]{cz_ajm}
P.~Corvaja and U.~Zannier.
\newblock On a general {T}hue's equation.
\newblock {\em Amer. J. Math.}, 126(5):1033--1055, 2004.

\bibitem[CZ04b]{cz_annals}
P.~Corvaja and U.~Zannier.
\newblock On integral points on surfaces.
\newblock {\em Ann. of Math. (2)}, 160(2):705--726, 2004.

\bibitem[CZ06a]{cz_ajm_add}
P.~Corvaja and U.~Zannier.
\newblock Addendum to: ``{O}n a general {T}hue's equation'' [{A}mer. {J}.
  {M}ath. {\bf 126} (2004), no. 5, 1033--1055; mr2089081].
\newblock {\em Amer. J. Math.}, 128(4):1057--1066, 2006.

\bibitem[CZ06b]{cz_imrn}
P.~Corvaja and U.~Zannier.
\newblock On the integral points on certain surfaces.
\newblock {\em Int. Math. Res. Not.}, pages Art. ID 98623, 20, 2006.

\bibitem[EF02]{ef_imrn}
J.-H. Evertse and R.~Ferretti.
\newblock Diophantine inequalities on projective varieties.
\newblock {\em Int. Math. Res. Not.}, (25):1295--1330, 2002.

\bibitem[EF08]{ef_festschrift}
J.-H. Evertse and R.~Ferretti.
\newblock A generalization of the {S}ubspace {T}heorem with polynomials of
  higher degree.
\newblock In {\em Diophantine approximation}, volume~16 of {\em Dev. Math.},
  pages 175--198. SpringerWienNewYork, Vienna, 2008.

\bibitem[Fer00]{ferretti_compositio}
R.~Ferretti.
\newblock Mumford's degree of contact and {D}iophantine approximations.
\newblock {\em Compositio Math.}, 121(3):247--262, 2000.

\bibitem[Gre75]{green_ajm}
M.~Green.
\newblock Some {P}icard theorems for holomorphic maps to algebraic varieties.
\newblock {\em Amer. J. Math.}, 97:43--75, 1975.

\bibitem[Har74]{hartshorne_bull_ams}
R.~Hartshorne.
\newblock Varieties of small codimension in projective space.
\newblock {\em Bull. Amer. Math. Soc.}, 80:1017--1032, 1974.

\bibitem[Lan83]{lang_fund_dioph_geom}
S.~Lang.
\newblock {\em Fundamentals of {D}iophantine geometry}.
\newblock Springer-Verlag, New York, 1983.

\bibitem[Lan87]{lang_book}
S.~Lang.
\newblock {\em Introduction to complex hyperbolic spaces}.
\newblock Springer-Verlag, New York, 1987.

\bibitem[Lar73]{larsen_invent}
M.~Larsen.
\newblock On the topology of complex projective manifolds.
\newblock {\em Invent. Math.}, 19:251--260, 1973.

\bibitem[Laz04]{PAGI}
R.~Lazarsfeld.
\newblock {\em Positivity in algebraic geometry. {I}}, volume~48 of {\em
  Ergebnisse der Mathematik und ihrer Grenzgebiete. 3. Folge. A Series of
  Modern Surveys in Mathematics [Results in Mathematics and Related Areas. 3rd
  Series. A Series of Modern Surveys in Mathematics]}.
\newblock Springer-Verlag, Berlin, 2004.
\newblock Classical setting: line bundles and linear series.

\bibitem[Lev09]{levin_annals}
A.~Levin.
\newblock Generalizations of {S}iegel's and {P}icard's theorems.
\newblock {\em Ann. of Math. (2)}, 170(2):609--655, 2009.

\bibitem[Mum77]{mumford_enseign}
D.~Mumford.
\newblock {\em Stability of projective varieties}.
\newblock L'Enseignement Math\'ematique, Geneva, 1977.
\newblock Lectures given at the ``Institut des Hautes {\'E}tudes
  Scientifiques'', Bures-sur-Yvette, March-April 1976, Monographie de
  l'Enseignement Math{\'e}matique, No. 24.

\bibitem[Ogu73]{ogus_annals}
A.~Ogus.
\newblock Local cohomological dimension of algebraic varieties.
\newblock {\em Ann. of Math. (2)}, 98:327--365, 1973.

\bibitem[Ru97]{ru_gen_form_smt_tams}
M.~Ru.
\newblock On a general form of the second main theorem.
\newblock {\em Trans. Amer. Math. Soc.}, 349(12):5093--5105, 1997.

\bibitem[Ru04]{ru_ajm}
M.~Ru.
\newblock A defect relation for holomorphic curves intersecting hypersurfaces.
\newblock {\em Amer. J. Math.}, 126(1):215--226, 2004.

\bibitem[Ru09]{ru_annals}
M.~Ru.
\newblock Holomorphic curves into algebraic varieties.
\newblock {\em Ann. of Math. (2)}, 169(1):255--267, 2009.

\bibitem[Voj89]{vojta_ajm_87}
P.~Vojta.
\newblock A refinement of {S}chmidt's subspace theorem.
\newblock {\em Amer. J. Math.}, 111(3):489--518, 1989.

\bibitem[Voj97]{vojta_smt_ajm_1997}
P.~Vojta.
\newblock On {C}artan's theorem and {C}artan's conjecture.
\newblock {\em Amer. J. Math.}, 119(1):1--17, 1997.

\end{thebibliography}
\end{document}